\documentclass[11pt]{amsart}
\usepackage{amscd}
\usepackage{amsfonts}
\usepackage{amssymb,latexsym}
\usepackage{enumerate}
\usepackage{amsmath}
\usepackage{mathrsfs}
\usepackage{amsthm}
\usepackage{hyperref}
\usepackage[square, comma, sort&compress, numbers]{natbib}
\usepackage{bm}
\usepackage{esint}
\usepackage{fancyhdr}
\usepackage{lastpage}
\usepackage{layout}
\usepackage{ulem}
\usepackage{extarrows}
\usepackage{geometry}
\newcommand{\ol}{\overline}

\newcommand{\pa}{\partial}
\newcommand{\vp}{\varphi}
\DeclareMathOperator\tr{tr}
\DeclareMathOperator\id{id}

\DeclareMathOperator\vol{vol}
\DeclareMathOperator\dvol{dvol}

\DeclareMathOperator\loc{loc}

\begin{document}
\newcounter{remark}
\newcounter{theor}
\setcounter{remark}{0}
\setcounter{theor}{1}
\newtheorem{claim}{Claim}[section]
\newtheorem{theorem}{Theorem}[section]
\newtheorem{lemma}{Lemma}[section]
\newtheorem{corollary}{Corollary}[section]
\newtheorem{corollarys}{Corollary}
\newtheorem{proposition}{Proposition}[section]
\newtheorem{question}{question}
\newtheorem{defn}{Definition}[section]
\newtheorem{examp}{Example}[section]
\newtheorem{assumption}{Assumption}[section]
\newtheorem{case}{Case}
\newtheorem*{mainassumption}{Main Assumption}
\newtheorem{rem}{Remark}[section]
\newtheorem*{theorem1}{Theorem}
\numberwithin{equation}{section}
\title{Prescribing Chern scalar curvatures on noncompact manifolds}
\author{Di Wu}
\address{Di Wu, School of Mathematics and Statistics, Nanjing University of Science and Technology, Nanjing 210094, People's Republic of China}
\email{wudi123@mail.ustc.edu.cn}
\author{Xi Zhang}
\address{Xi Zhang, School of Mathematics and Statistics, Nanjing University of Science and Technology, Nanjing 210094, People's Republic of China}
\email{mathzx@ustc.edu.cn}
\subjclass[]{35J60, 53A30, 53C55}
\keywords{Noncompact manifold, Non-K\"{a}hler metric, Prescribed Chern scalar curvature, Kazdan-Warner type equation, Multiplicity of solutions.}
\thanks{The research was supported by the National Key R\&D Program of China 2020YFA0713100. Both authors are partially supported by NSF in China No.12141104. The first author is also supported by the Jiangsu Funding Program for Excellent Postdoctoral Talent 2022ZB282.}
\maketitle
\begin{abstract}
In this paper, we investigate the noncompact prescribed Chern scalar curvature problem which reduces to solve a Kazdan-Warner type equation on noncompact non-K\"{a}hler manifolds. By introducing an analytic condition on noncompact manifolds, we establish related existence results. As its another application, we further give a new proof of a classical multiplicity theorem of W.M. Ni \cite{Ni1982}.
\end{abstract}
\section{Introduction}
\par As a Hermitian analogue of the Yamabe problem \cite{Ya1960}, Angella-Calamai-Spotti \cite{ACS2017} initiated the study on the existence of non-K\"{a}hler metrics with constant Chern scalar curvatures in a given conformal class, which was called the Chern-Yamabe problem. More generally, the prescribed Chern scalar curvature problem consists of finding a conformally equivalent Hermitian metric such that whose Chern scalar curvature equals to a previously given function. By a classical result of Gauduchon \cite{Ga1977}, there exists on any compact complex manifold $X$ a Hermitian metric $g$ in every conformal class such that the fundamental form $\omega_g$ satisfies the now called Gauduchon condition $\pa\ol\pa\omega_g^{\dim_{\mathbb{C}}X-1}=0$. The problem in compact case, roughly speaking, can be divided into three cases according to the sign of the Gauduchon degree as an invariant of the conformal class, we refer readers to \cite{ACS2017,AF2022,CZ2020,Fu2022,Ho2021,HS2021,LM2018,Yu2023} for recent progress. It is worth pointing out that the positive case is the most difficult since the PDE loses its good analytic properties.
\par We initiate to study the prescribed Chern scalar curvature problem on noncompact manifolds. For this, we introduce the following assumption and unless indicated explicitly, we always assume it is satisfied throughout this paper.
\begin{mainassumption}\label{assump1}
Assume $(X,g)$ is a $n$-dimensional (not necessarily compact or complete)Gauduchon manifold admitting a nonnegative integrable $\phi_X\in C^\infty(X)$ such that if $f\in C^\infty(X)$ is nonnegative and bounded with $\sqrt{-1}\Lambda_{\omega_g}\pa\ol\pa f\geq-A\phi_X$ for a positive constant $A$, then
\begin{equation}\begin{split}
\sup\limits_{X}f\leq C_A(1+\int_Xf\phi_X\dvol_{g}),
\end{split}\end{equation}
where $C_A$ is a positive constant depending on $A$ and $\dvol_g$ denotes the volume element.
\end{mainassumption}
We first provide an existence result on the prescribed Chern scalar curvature problem on noncompact non-K\"{a}hler manifolds. In fact, we prove
\begin{theorem}[See Theorem \ref{KWthm}]\label{thm1}
Suppose that the Chern scalar curvature $S_g^{Ch}$ satisfies $\int_XS^{Ch}_g\dvol_g<0$ and $|S^{Ch}_g|\leq\Lambda\phi_X$ for a constant $\Lambda$. Let $h$ be a smooth nonpositive and nonzero function on $X$ with $|h|\leq\Lambda\phi_X$, then there exists a bounded and conformally equivalent Hermitian metric $\tilde{g}$ such that whose Chern scalar curvature $S_{\tilde{g}}^{Ch}=h$.
\end{theorem}
\begin{rem}
An important feature of Theorem \ref{thm1} is that no additional hypothesis required other than our Main Assumption. This makes the result applicable on prescribing Chern scalar curvatures on quasi-compact complex manifolds(see Corollary \ref{thm2}) and solving the Kazdan-Warner type equation $(\ref{KW1})$ on $\mathbb{C}$(see Proposition \ref{vortex1}), as well as on $M\times\mathbb{C}$ for any compact complex manifold $M$(see Proposition \ref{vortex2}).
\end{rem}
In compact case, Theorem \ref{thm1} recovers \cite[Theorem 2.5]{Fu2022} and \cite[Proposition 2.1]{LY2005} which were proved via the standard method of super-sub solutions. The same conclusion was also obtained in \cite[Theorem 1.1]{Ho2021} and \cite[Corollary 1.3]{Yu2023} via the flow method, where the extra balanced condition $d\omega_g^{n-1}=0$ is further assumed. In noncompact case, to prove Theorem \ref{thm1}, we employ the approach of a combination of heat flow and continuity methods to solve the following Kazdan-Warner type equation
\begin{equation}\begin{split}\label{KW1}
\sqrt{-1}\Lambda_{\omega_g}\pa\ol\pa u+he^{u}=f,
\end{split}\end{equation}
on noncompact manifolds, where $f$, $h$ are two functions satisfying certain conditions. The key point in our argument is the uniform zeroth order estimate on noncompact spaces. Arising from the basic geometric problem on prescribing Gaussian curvatures on Riemannian surfaces, the classical Kazdan-Warner equation \cite{KW1974a} takes the form
\begin{equation}\begin{split}\label{KW2}
\Delta_g u+he^{2u}=f,
\end{split}\end{equation}
where $\Delta_g$ denotes the Beltrami-Laplacian. Up to a constant multiple, $(\ref{KW1})$ coincides with $(\ref{KW2})$ for balanced metrics and they generally differ by a gradient term. It is also mentioned that our method can be further utilized for improving the results in \cite{WZ2008,WZ2021}.
\par Recall the singular Yamabe problem: given a compact Riemannian manifold $(M,g)$ and a closed subset $\Sigma$, find a conformally equivalent metric $\tilde{g}$ on $M\setminus\Sigma$ which has constant scalar curvature. The Chern-Yamabe problem generalizes naturally to the
singular case as well and Corollary \ref{thm2} below shows the problem can be solved in certain cases. As mentioned in \cite{DL2017}, it would be interesting to find non-K\"{a}hler metrics which simultaneously solve both the singular Yamabe and singular Chern-Yamabe problems.
\begin{corollary}[See Proposition \ref{assumption}]\label{thm2}
Suppose that $X=\ol{X}\setminus\Sigma$ is the complement of a complex analytic subset $\Sigma$ in a compact complex manifold $\ol{X}$ with complex codimension at least two, $g=\ol{g}|_{X}$ for a Gauduchon metric $\ol{g}$ on $\ol{X}$ and $\int_XS^{Ch}_g\dvol_g<0$. Let $h$ be a smooth bounded nonpositive and nonzero function on $X$, then there exists a bounded and conformally equivalent Hermitian metric $\tilde{g}$ such that $S^{Ch}_{\tilde{g}}=h$.
\end{corollary}
Next we show Theorem \ref{thm1} can be also used to detect the multiplicity of conformal metrics with prescribed curvatures on the entire plane, which is equivalent to study the non-uniqueness of solutions to the following semi-linear elliptic equation on $\mathbb{R}^2$:
\begin{equation}\begin{split}\label{eqR2}
\Delta u+Ke^{2u}=0,
\end{split}\end{equation}
where $\Delta=\pa^2_x+\pa^2_y$ and $K\in C^\infty(\mathbb{R}^2)$ is the candidate curvature function. The first result concerning $(\ref{eqR2})$ seems due to Ahlfors \cite{Ah1938} and he proved that there is no entire solution if $K$ is a negative constant. Sattinger \cite{Sa1972} obtained a more general unsolvable result when $K\leq0$ and $|K|\geq C|x|^{-s}$ at infinity for two positive constants $C$ and $s\leq2$.
\par By explicitly constructing radially symmetric sub-solutions and super-solutions together with the method of super-sub solutions for the entire plane, the first nontrivial existence result on (multiple)solutions to $(\ref{eqR2})$ was founded by Ni.
\begin{theorem}[Ni \cite{Ni1982}, Theorem 1.3]\label{Nithm}
If $K\in C^\infty(\mathbb{R}^2)$ is nonpositive and nonzero with decay $|K|\leq C|z|^{-s}$ at infinity, for two constants $s>2$ and $C$, then $(\ref{eqR2})$ possesses infinitely many solutions on $\mathbb{R}^2$ and each of them has the following logarithmic growth at infinity for two constants $C_1$ and $C_2$,
\begin{equation}\begin{split}\label{Niformsofsolutions}
k\log|z|-C_1\leq u_k\leq k\log|z|+C_2.
\end{split}\end{equation}
\end{theorem}
\begin{rem}
A great deal of work has been devoted to understanding the large variety of solutions to $(\ref{eqR2})$ on $\mathbb{R}^2$ since then, one may consult \cite{Av1986,CK2000,CL1991,CL1993,CL1997,CL2000,
CN1991a,CN1991b,KW1974b,Li2000,Mc1984,Mc1985} which is by far an incomplete list of important contributions.
\end{rem}
\par Below comes our multiplicity result.
\begin{theorem}[See Proposition \ref{Aprop2}]\label{thm3}
If $K\in C^\infty(\mathbb{R}^2)$ is nonpositive and nonzero with decay $|K|\leq\Lambda(1+|z|^2)^{-l}$ for two constants $l>1$ and $\Lambda$, then $(\ref{eqR2})$ possesses infinitely many solutions on $\mathbb{R}^2$ and they have the forms
\begin{equation}\begin{split}
2u_k=v_k+k\log(1+|z|^2),
\end{split}\end{equation}
with $v_k\in C^\infty(\mathbb{R}^2)\cap L^\infty(\mathbb{R}^2)$, $|dv_k|\in L^2(\mathbb{R}^2)$ and $k$ constant, subject to
\begin{equation}\begin{split}
k\in[l-2,l-1)\cap(0,\infty).
\end{split}\end{equation}
On the other hand, any solution $u$ of the form
\begin{equation}\begin{split}\label{formsofsolutions}
2u=v+k\log(1+|z|^2),
\end{split}\end{equation}
with $v\in C^\infty(\mathbb{R}^2)\cap L^\infty(\mathbb{R}^2)$ and constant $k\in[l-2,l-1)\cap(0,\infty)$, is exactly $u_k$.
\end{theorem}
\begin{rem}
The method to prove Theorem \ref{thm3} is different in spirit from those already exist in the literature and the strategy seems to us more geometrically. Indeed, we shall transform the issue of the multiplicity of solutions into constructing infinitely many metrics meeting our Main Assumption so that Theorem \ref{thm1} can be applied. 
\end{rem}
To recover Ni's Theorem \ref{Nithm} from Theorem \ref{thm3}, we may write $s=2l$ for $l>1$ and take large constant $B$(for which $|K|\leq C|z|^{-s}$ once $|z|\geq B$),
\begin{equation}\begin{split}
\Lambda=\max\{2C^l, \max\limits_{|z|\leq B}|K|(1+|B|^2)^l\},
\end{split}\end{equation}
then it is easy to see $|K|\leq\Lambda(1+|z|^2)^{-l}$ and hence Theorem \ref{thm3} guarantees the multiplicity of solutions with each of them satisfying
\begin{equation}\begin{split}
\frac{k}{2}\log(1+|z|^2)-\frac{1}{2}\sup\limits_{\mathbb{R}^2}|v_k|\leq u_k\leq\frac{1}{2}\sup\limits_{\mathbb{R}^2}|v_k|+\frac{k}{2}\log(1+|z|^2),
\end{split}\end{equation}
and then the logarithmic growth $(\ref{Niformsofsolutions})$ at infinity follows.
\par Finally, our results also solve the vortex equation on holomorphic line bundles over some noncompact manifolds, see Section 4. Note the vortex equation was introduced by Bradlow \cite{Br1990}, the solvability was characterized for closed manifolds via the result in \cite{KW1974a} and its higher rank analogue was studied in \cite{Br1991} which involves a stability-like criterion.
\par The rest of this paper is organized as follows. In Section 2, we first recall a few basic facts related the prescribed Chern scalar curvature problem. Then we establish an existence result and illustrate it applies to some quasi-compact complex manifolds. In Section 3, we first construct certain metrics on entire plane meeting the Main Assumption and then prove multiplicity results on the prescribed curvature equation. In Section 4, we apply the results to solve the vortex equation on holomorphic line bundles.
\section{Noncompact prescribed Chern scalar curvature problem}
Let $X$ be a $n$-dimensional complex manifold with the natural complex structure
\begin{equation}\begin{split}
J:TX\rightarrow TX,\ J^2=-\id.
\end{split}\end{equation}
In the presence of $J$, the tangent bundle $TX$ is in particular a complex vector bundle which will be denoted by $(TX,J)$. Since $J^2=-\id$, we have
\begin{equation}\begin{split}
TX\otimes\mathbb{C}=T^{1,0}X\oplus T^{0,1}X,
\end{split}\end{equation}
where $T^{1,0}X$ and $T^{1,0}X$ are the eigenspaces of eigenvalues $\sqrt{-1}$ and $-\sqrt{-1}$ respectively. In fact, for any $x\in X$ we can find vectors $\{e_i,Je_i\}_{i=1}^{n}$ forming a basis of $T_xX$ and then $T^{1,0}_xX$($T^{0,1}_xX$) is generated by $\{e_i-\sqrt{-1}Je_i\}_{i=1}^n$($\{e_i+\sqrt{-1}Je_i\}_{i=1}^n$). Let us write
\begin{equation}\begin{split}
A^{p,q}(X)=\wedge^p(T^{1,0}X)^\ast\otimes\wedge^q(T^{0,1}X)^\ast.
\end{split}\end{equation}
Then the usual exterior differential operator can be written as $d=\pa+\ol\pa$ with
\begin{equation}\begin{split}
\pa: A^{p,q}(X)\rightarrow A^{p+1,q}(X),\
\ol\pa: A^{p,q}(X)\rightarrow A^{p,q+1}(X).
\end{split}\end{equation}
We shall often identify $TX$ with $T^{1,0}X$ via
\begin{equation}\begin{split}\label{identification}
e_i\leftrightarrow\frac{1}{\sqrt{2}}(e_i-\sqrt{-1}Je_i),\  Je_i\leftrightarrow\frac{\sqrt{-1}}{\sqrt{2}}(e_i+\sqrt{-1}Je_i),\
i=1,...,n.
\end{split}\end{equation}
\par Given a Hermitian metric $g$ on $(X,J)$ which means a Riemannian metric satisfying $g(a,b)=g(Ja,Jb)$ for any $a,b\in TX$, we can equip $(TX,J)$ with a Hermitian structure
\begin{equation}\begin{split}\label{Hg}
H_g=g-\sqrt{-1}\omega_g,
\end{split}\end{equation}
where $\omega_g(a,b)=g(Ja,b)$ is the associated fundamental form. Using $(\ref{identification})$ and $(\ref{Hg})$, we know that $T^{1,0}X$ inherits a Hermitian structure given by
\begin{equation}\begin{split}
H_g(E_i,E_j)=g(E_i,\ol{E_j}),
\end{split}\end{equation}
where $E_i=e_i-\sqrt{-1}Je_i$, $\ol{E}_j=e_j+\sqrt{-1}Je_j$ and $g$ has been extended by $\mathbb{C}$-linearity. Hence $(T^{1,0}X,H_g)$ is Hermitian vector bundle equipped with a holomorphic structure induced by $J$, there is a unique connection $\nabla$(called the Chern connection) on $T^{1,0}X$ which is compatible with $H_g$ and also the  holomorphic structure.
\par In the local holomorphic coordinate, we write
\begin{equation}\begin{split}
\omega_g=\sqrt{-1}g_{i\ol{j}}dz^i\wedge d\ol{z}^{j},\
g_{i\ol{j}}=g(\frac{\pa}{\pa z^i},\frac{\pa}{\pa\ol{z}^j}).
\end{split}\end{equation}
The Chern connection $\nabla$ is characterized by the connection $1$-form
\begin{equation}\begin{split}
A=\pa g\cdot g^{-1}
\end{split}\end{equation}
where $g=(g_{i\ol{j}})_{n\times n}$ is regarded as a Hermitian matrix. Precisely, we have
\begin{equation}\begin{split}
\nabla\frac{\pa}{\pa z_i}=\frac{\pa g_{i\ol{l}}}{\pa z^k}g^{\ol{l}j}\frac{\pa}{\pa z^j}\otimes dz^k.
\end{split}\end{equation}
From this, it is known that $\nabla$ being torsion-free if and only of the Hermitian metric $g$ is K\"{a}hler(that is, $d\omega_g=0$). The curvature $R_\nabla$ is given by the curvature $2$-form
\begin{equation}\begin{split}
\Omega=dA-A\wedge A=\ol\pa(\pa g\cdot g^{-1}).
\end{split}\end{equation}
Precisely, the curvature tensor has components
\begin{equation}\begin{split}\label{R}
R_{i\ol{j}k\ol{l}}=-\frac{\pa^2g_{k\ol{l}}}{\pa z^i\pa\ol{z}^j}
+g^{\ol{q}p}\frac{\pa g_{p\ol{l}}}{\pa\ol{z}^j}\frac{\pa g_{k\ol{q}}}{\pa z^i},
\end{split}\end{equation}
where $R_{i\ol{j}k\ol{l}}=g([\nabla_{\frac{\pa}{\pa z^i}},\nabla_{\frac{\pa}{\pa\ol{z}^j}}]\frac{\pa}{\pa z^k}-\nabla_{[\frac{\pa}{\pa z^i},\frac{\pa}{\pa z^j}]}\frac{\pa}{\pa z^k},\frac{\pa}{\pa\ol{z}^l})$. Then the first Chern-Ricci form
\begin{equation}\begin{split}
Ric^{(1)}=\frac{\sqrt{-1}}{2\pi}\tr\Omega=\frac{\sqrt{-1}}{2\pi}R_{i\ol{j}}^{(1)}dz^i\wedge d\ol{z}^j,\
R_{i\ol{j}}^{(1)}=-\frac{\pa^2\log\det g}{\pa z^i\pa\ol{z}^j},
\end{split}\end{equation}
represents the first Chern class $c_1(T^{1,0}X)$.
The Chern scalar curvature is defined by
\begin{equation}\begin{split}\label{S}
S^{Ch}_g=g^{\ol{j}i}R_{i\ol{j}}^{(1)}
=-\sqrt{-1}\Lambda_{\omega_g}\pa\ol\pa\log\det g,
\end{split}\end{equation}
where $\Lambda_{\omega_g}$ denotes the operator obtained as the adjoint of the multiplication of $\omega_g$.
\par Taking a conformal change $\tilde{g}=e^{\frac{u}{n}}g$, the Chern scalar curvature changes as follows,
\begin{equation}\begin{split}\label{conformal}
S^{Ch}_{\tilde{g}}=e^{-\frac{u}{n}}(S^{Ch}_g-\sqrt{-1}\Lambda_{\omega_g}\pa\ol\pa u).
\end{split}\end{equation}
Given a smooth function $h$ on $X$, $S^{Ch}_{\tilde{g}}=h$ if and only if
\begin{equation}\begin{split}
\sqrt{-1}\Lambda_{\omega_g}\pa\ol\pa u+he^{\frac{u}{n}}=S^{Ch}_g.
\end{split}\end{equation}
\par Now Theorem \ref{thm1} follows from Theorem \ref{KWthm} below.
\begin{theorem}\label{KWthm}
Let $f$ and $h$ be two smooth nonzero functions on $X$ such that
\begin{equation}\begin{split}
\int_Xf\dvol_g<0,\ |f|\leq\Lambda\phi_X,\ -\Lambda\phi_X\leq h\leq0,
\end{split}\end{equation}
for a constant $\Lambda$, then there exists a bounded smooth function $u$ such that $|du|\in L^2(X,g)$ and it solves the following Kazdan-Warner type equation on whole $X$:
\begin{equation}\begin{split}
\sqrt{-1}\Lambda_{\omega_g}\pa\ol\pa u+he^{u}=f.
\end{split}\end{equation}
\end{theorem}
\begin{proof}
Let $\{X_j\}$ be an exhaustion series of closed submanifolds with nonempty boundaries(see \cite[Lemma 2.31]{Mo2020}). For any $X_j$ and $\epsilon\in[0,1]$, we consider
\begin{equation}\begin{split}\label{heatflow}
\left\{ \begin{array}{ll}
\frac{\pa\ol{u}_{\epsilon,j}}{\pa t}=\sqrt{-1}\Lambda_{\omega_g}\pa\ol\pa \ol{u}_{\epsilon,j}+he^{\ol{u}_{\epsilon,j}}-f-\epsilon\ol{u}_{\epsilon,j},\\
\ol{u}_{\epsilon,j}(0)=0,\ \ol{u}_{\epsilon,j}|_{\pa X_j}=0.
\end{array}\right.
\end{split}\end{equation}
It is not hard to show the convergence at infinity and hence it yields a solution $u_{\epsilon,j}$ to the Dirichlet problem(see \cite{WZ2021}):
\begin{equation}\begin{split}
\sqrt{-1}\Lambda_{\omega_g}\pa\ol\pa u_{\epsilon,j}+he^{u_{\epsilon,j}}=f+\epsilon u_{\epsilon,j},\ u_{\epsilon,j}|_{\pa X_j}=0,\ \epsilon\in[0,1].
\end{split}\end{equation}
\par Since $x\leq xe^{x}$ for $x\in\mathbb{R}$ and $h\leq0$, we compute
\begin{equation}\begin{split}
\sqrt{-1}\Lambda_{\omega_g}\pa\ol\pa u_{\epsilon,j}^2
&=2u_{\epsilon,j}\sqrt{-1}\Lambda_{\omega_g}\pa\ol\pa u_{\epsilon,j}+2|\pa u_{\epsilon,j}|^2
\\&\geq2u_{\epsilon,j}(f+\epsilon u_{\epsilon,j}-he^{u_{\epsilon,j}})
\\&\geq2u_{\epsilon,j}(f+\epsilon u_{\epsilon,j}-h)
\\&\geq2|u_{\epsilon,j}|(\epsilon|u_{\epsilon,j}|-|f|-|h|),
\end{split}\end{equation}
and therefore
\begin{equation}\begin{split}\label{zerothestimate}
|u_{\epsilon,j}|\leq\frac{||f||_{L^\infty}+||h||_{L^\infty}}{\epsilon}\leq\frac{2\Lambda}{\epsilon}||\phi_X||_{L^\infty}.
\end{split}\end{equation}
By the Gauduchon condition and $h\leq0$, integrating by part gives
\begin{equation}\begin{split}\label{L2estimate}
\int_{X_j}|du_{\epsilon,j}|^2\dvol_g
&=2\int_{X_j}\sqrt{-1}\Lambda_{\omega_g}(\pa u_{\epsilon,j}\wedge\ol\pa u_{\epsilon,j})\dvol_g
\\&=2\int_{X_j}\sqrt{-1}\pa u_{\epsilon,j}\wedge\ol\pa u_{\epsilon,j}\wedge\frac{\omega_g^{n-1}}{(n-1)!}
\\&=\int_{X_j}-2\sqrt{-1}u_{\epsilon,j}\wedge\pa\ol\pa u_{\epsilon,j}\wedge\frac{\omega_g^{n-1}}{(n-1)!}
+\sqrt{-1}\ol\pa u^2_{\epsilon,j}\wedge\frac{\pa\omega_g^{n-1}}{(n-1)!}
\\&=\int_{X_j}-2u_{\epsilon,j}\sqrt{-1}\Lambda_{\omega_g}\pa\ol\pa u_{\epsilon,j}\dvol_g
+\sqrt{-1}\ol\pa(u^2_{\epsilon,j}\wedge\frac{\pa\omega_g^{n-1}}{(n-1)!})
\\&=2\int_{X_j}u_{\epsilon,j}(he^{u_{\epsilon,j}}-f-\epsilon u_{\epsilon,j})\dvol_g
\\&\leq2\int_{X_j}u_{\epsilon,j}(h-f)\dvol_g
\\&\leq\frac{8\Lambda^2}{\epsilon}||\phi_X||_{L^\infty}\int_X\phi_X\dvol_g.
\end{split}\end{equation}
Thanks to the zeroth order estimate $(\ref{zerothestimate})$ and standard elliptic estimates, for any positive $\epsilon$ we have a subsequence $\{u_{\epsilon,j_\epsilon}\}_{j_\epsilon\in\mathbb{N}}$ converging in $C^\infty_{\loc}$-topology to $u_{\epsilon}$ solving the following equation on whole $X$:
\begin{equation}\begin{split}
\sqrt{-1}\Lambda_\omega\pa\ol\pa u_{\epsilon}+he^{u_{\epsilon}}=f+\epsilon u_{\epsilon},
\end{split}\end{equation}
and satisfying
\begin{equation}\begin{split}
|u_{\epsilon}|\leq\frac{2\Lambda}{\epsilon}||\phi_X||_{L^\infty},
\end{split}\end{equation}
\begin{equation}\begin{split} 
||du_{\epsilon}||_{L^2}
\leq\left(\frac{8\Lambda^2}{\epsilon}||\phi_X||_{L^\infty}\int_X\phi_X\dvol_g\right)^{\frac{1}{2}}.
\end{split}\end{equation}
\par We shall show that $\{u_\epsilon\}_{\epsilon>0}$ must subconverge to some $u_0$ solving desired equation. It suffices to prove the uniform zeroth order estimate. To this end, we compute
\begin{equation}\begin{split}
\sqrt{-1}\Lambda_{\omega_g}\pa\ol\pa\log(e^{u_\epsilon}+e^{-u_\epsilon})
&=\frac{\sqrt{-1}\Lambda_{\omega_g}\pa\ol\pa(e^{u_\epsilon}+e^{-u_\epsilon})}{e^{u_\epsilon}+e^{-u_\epsilon}}-\frac{|\pa(e^{u_\epsilon}+e^{-u_\epsilon})|^2}{(e^{u_\epsilon}+e^{-u_\epsilon})^2}
\\&=\frac{e^{u_\epsilon}-e^{-u_\epsilon}}{e^{u_\epsilon}+e^{-u_\epsilon}}\sqrt{-1}\Lambda_{\omega_g}\pa\ol\pa u_\epsilon
+\frac{e^{u_\epsilon}+e^{-u_\epsilon}}{e^{u_\epsilon}+e^{-u_\epsilon}}|\pa u_\epsilon|^2
\\&-\frac{(e^{u_\epsilon}-e^{-u_\epsilon})^2}{(e^{u_\epsilon}+e^{-u_\epsilon})^2}|\pa u_\epsilon|^2
\\&\geq\frac{e^{u_\epsilon}-e^{-u_\epsilon}}{e^{u_\epsilon}+e^{-u_\epsilon}}(f+\epsilon u_{\epsilon}-he^{u_{\epsilon}})
\\&=\frac{e^{u_\epsilon}-e^{-u_\epsilon}}{e^{u_\epsilon}+e^{-u_\epsilon}}\left(f-h+\epsilon u_{\epsilon}-h(e^{u_{\epsilon}}-1)\right)
\\&\geq\frac{e^{u_\epsilon}-e^{-u_\epsilon}}{e^{u_\epsilon}+e^{-u_\epsilon}}(f-h)
\\&\geq-|f|-|h|
\\&\geq-2\Lambda\phi_X,
\end{split}\end{equation}
where we have also used $h\leq0$, and therefore the Main Assumption implies
\begin{equation}\begin{split}
|u_\epsilon|
&\leq\log(e^{u_\epsilon}+e^{-u_\epsilon})
\\&\leq C_1(1+\int_X|\log(e^{u_\epsilon}+e^{-u_\epsilon})|\phi_X\dvol_g)
\\&\leq C_1\left(1+\int_X(|u_\epsilon|+\log2)\phi_X\dvol_g\right)
\\&\leq C_2(1+\int_X|u_\epsilon|\phi_X\dvol_g)
\\&\leq C_2(1+\frac{2\Lambda||\phi_X||_{L^\infty}}{\epsilon}\int_X\phi_X\dvol_g),
\end{split}\end{equation}
where $C_1$, $C_2$ are independent of $\epsilon$. If $\{u_\epsilon\}_{\epsilon>0}$ are not uniformly bounded, there is a subsequence $l_i\rightarrow\infty$, where for simplicity $l_i=\int_X|u_{i}|\phi_X\dvol_g>1$ and $u_i=u_{\epsilon_i}$. Then
\begin{equation}\begin{split}
|v_i|\leq C_3,\ \int_X|v_i|\phi_X\dvol_g=1,
\end{split}\end{equation}
where $v_i=l_i^{-1}u_i$ and $C_3$ is independent of $\epsilon$. Recall that
\begin{equation}\begin{split}
\int_{X_{j_i}}|du_{i,j_i}|^2\dvol_g
&=2\int_{X_{j_i}}u_{i,j_i}(he^{u_{i,j_i}}-f-\epsilon u_{i,j_i})\dvol_g,
\end{split}\end{equation}
where $u_{i,j_i}=u_{\epsilon_i,j_{\epsilon_i}}$ and $X_{j_i}=X_{j_{\epsilon_i}}$. It follows
\begin{equation}\begin{split}
\int_{X_{j_i}}|dv_{i,j_i}|^2\dvol_g
&\leq2l_{i,j_i}^{-1}\int_{X_{j_i}}v_{i,j_i}(he^{u_{i,j_i}}-f)\dvol_g,
\end{split}\end{equation}
where $v_{i,j_i}=l_{i,j_i}^{-1}u_{i,j_i}$ and $l_{i,j_i}=\int_{X_{j_i}}|u_{i,j_i}|\phi_X\dvol_g$.
We may assume $l_{i,j_i}\geq1$ for any $i$ and $j_i$. Note for any $\alpha>0$, we have for $\tilde{j}_i\geq j_i>>1$,
\begin{equation}\begin{split}
\int_{X_{\tilde{j}_i}-X_{j_i}}(|h|+|f|)\dvol_g\leq\alpha,
\end{split}\end{equation}
which yields
\begin{equation}\begin{split}
\int_{X_{j_i}}|dv_{i,\tilde{j}_i}|^2\dvol_g
&\leq\int_{X_{\tilde{j}_i}}|dv_{i,\tilde{j}_i}|^2\dvol_g
\\&\leq2l_{i,\tilde{j}_i}^{-1}(\int_{X_{\tilde{j}_i}-X_{j_i}}+\int_{X_{j_i}})v_{i,\tilde{j}_i}(he^{u_{i,\tilde{j}_i}}-f)\dvol_g
\\&\leq2C_3(e^{\frac{2\Lambda}{\epsilon_i}||\phi_X||_{L^\infty}}+1)\int_{X_{\tilde{j}_i}-X_{j_i}}(|h|+|f|)\dvol_g
\\&+2l_{i,\tilde{j}_i}^{-1}\int_{X_{j_i}}v_{i,\tilde{j}_i}(he^{u_{i,\tilde{j}_i}}-f)\dvol_g
\\&\leq2C_3(e^{\frac{2\Lambda}{\epsilon_i}||\phi_X||_{L^\infty}}+1)\alpha
+2l_{i,\tilde{j}_i}^{-1}\int_{X_{j_i}}v_{i,\tilde{j}_i}(he^{u_{i,\tilde{j}_i}}-f)\dvol_g.
\end{split}\end{equation}
Taking $\tilde{j}_i\rightarrow\infty$, $j_i\rightarrow\infty$ and noting $\alpha$ is arbitrary, we arrive at
\begin{equation}\begin{split}\label{inequality}
\int_{X}|dv_{i}|^2\dvol_g
&\leq2l_{i}^{-1}\int_{X}v_{i}(he^{u_{i}}-f)\dvol_g
\\&\leq2l_{i}^{-1}\int_{X}v_{i}(h-f)\dvol_g
\\&\leq4\Lambda C_3l_{i}^{-1}\int_X\phi_X\dvol_g.
\end{split}\end{equation}
Up to a subsequence, it is known that $v_i\rightarrow v_0$ weakly in $L_1^2$-topology and $v_0$ equals to a constant almost everywhere. For any $\beta>0$ we have for $j>>1$,
\begin{equation}\begin{split}
1-\beta
&\leq(\int_X-\int_{X-X_j})|v_i|\phi_X\dvol_g
\\&=\int_{X_j}|v_i|\phi_X\dvol_g
\\&\leq1,
\end{split}\end{equation}
\begin{equation}\begin{split}
\int_{X_{j}}|v_{i}|\phi_X\dvol_g
&\xrightarrow{i\rightarrow\infty}\int_{X_j}|v_{0}|\phi_X\dvol_g
\\&\xrightarrow{j\rightarrow\infty}\int_X|v_0|\phi_X\dvol_g.
\end{split}\end{equation}
It follows $\int_X|v_0|\phi_X\dvol_g=1$ and $v_0\neq0$. $(\ref{inequality})$ and $xe^{x}\geq-e^{-1}$ for $x\in\mathbb{R}$ imply
\begin{equation}\begin{split}
\int_Xfv_0\dvol_g
&=\lim\limits_{i\rightarrow\infty}\int_Xfv_i\dvol_g
\\&\leq\lim\limits_{i\rightarrow\infty}\int_Xhv_ie^{l_iv_i}\dvol_g
\\&\leq-\lim\limits_{i\rightarrow\infty}e^{-1}l_i^{-1}\int_Xh\dvol_g
\\&\leq-\lim\limits_{i\rightarrow\infty}\Lambda e^{-1}l_i^{-1}\int_X\phi_X\dvol_g
\\&=0,
\end{split}\end{equation}
we conclude from above that $v_0>0$ since $\int_Xf\dvol_g<0$. Furthermore, let us choose a small $\gamma\in(0,v_0)$ and a nonnegative smooth function $\eta_\gamma$ such that 
\begin{equation}\begin{split}
\eta_\gamma(x)=\left\{ \begin{array}{ll}
1,\ x\geq v_0,\\
0,\ x\leq\gamma.
\end{array}\right.
\end{split}\end{equation}
For any $x\in\mathbb{R}^+$ and $i>>1$ we have
\begin{equation}\begin{split}
\eta_\gamma(x)\leq x l_i^{-2}e^{xl_i},
\end{split}\end{equation}
where we have used $y^{-2}e^{xy}\rightarrow\infty$ as $y\rightarrow\infty$. It follows for large $i$ that $v_i>0$ and thus
\begin{equation}\begin{split}
\int_Xh\eta_\gamma(v_i)\dvol_g
&\geq l_i^{-2}\int_Xhv_ie^{l_iv_i}\dvol_g
\\&\geq l_i^{-2}\int_Xv_if\dvol_g.
\end{split}\end{equation}
Next for any compact subset $\Omega\subset X$, we may assume $v_i\rightarrow v_0$ strongly in $L^2(\Omega)$ and
\begin{equation}\begin{split}
\int_\Omega h\dvol_g
&=\int_\Omega\eta_\gamma(v_0)h\dvol_g
\\&=\lim\limits_{i\rightarrow\infty}\int_\Omega\eta_\gamma(v_i)h\dvol_g
+\lim\limits_{i\rightarrow\infty}\int_\Omega(\eta_\gamma(v_0)-\eta_\gamma(v_i))h\dvol_g
\\&\geq\lim\limits_{i\rightarrow\infty}\int_X\eta_\gamma(v_i)h\dvol_g
\\&\geq\lim\limits_{i\rightarrow\infty}l_i^{-2}\int_Xv_if\dvol_g
\\&\geq-\lim\limits_{i\rightarrow\infty}\Lambda C_3l_i^{-2}\int_X\phi_X\dvol_g
\\&=0.
\end{split}\end{equation}
This is impossible since $h$ is nonpositive and not identically zero.
\end{proof}
Corollary \ref{thm2} is a consequence of Theorem \ref{KWthm} and the following
\begin{proposition}\label{assumption}
Suppose that $X=\ol{X}\setminus\Sigma$ is the complement of a complex analytic subset $\Sigma$ in a compact complex manifold $\ol{X}$ with complex codimension at least two and $g=\ol{g}|_{X}$ for a Gauduchon metric $\ol{g}$ on $\ol{X}$, then $(X,g)$ satisfies the Main Assumption.
\end{proposition}
\begin{proof}
We take $\phi_X=1$ and may choose a family of cut-off functions $\{\eta_\delta\}_{\delta>0}$ such that $\eta_\delta=1$ on $\ol{X}\setminus B_{2\delta}(\Sigma)$, $\eta=0$ on $B_{\delta}(\Sigma)$(where $B_\delta(\Sigma)=\{x\in\ol{X}, d(x,\Sigma)<\delta\}$) and
\begin{equation}\begin{split}
\lim\limits_{\delta\rightarrow0}\int_{\ol{X}}|d\eta_\delta|^2\dvol_{\ol{g}}=O(\delta^{-2}\delta^{4})=0.
\end{split}\end{equation}
Assume $f\in C^\infty(X)$ is nonnegative and bounded with $\sqrt{-1}\Lambda_{\omega_g}\pa\ol\pa f\geq-A$ for a positive constant $A$. We set $\tilde{f}=f+1$ so that $\sqrt{-1}\Lambda_{\omega_g}\pa\ol\pa\tilde{f}\geq-A\tilde{f}$, then for any $q\geq1$,
\begin{equation}\begin{split}
-A\int_{\ol{X}}\eta_\delta^2\tilde{f}^{q+1}\omega_{\ol{g}}^{n}
&\leq\int_{\ol{X}}n\eta_\delta^2\tilde{f}^{q}\sqrt{-1}\pa\ol\pa\tilde{f}\wedge\omega^{n-1}_{\ol{g}}
\\&=n\int_{\ol{X}}-\sqrt{-1}\pa(\eta_\delta^2\tilde{f}^{q})\wedge\ol\pa\tilde{f}\wedge\omega^{n-1}_{\ol{g}}
+\sqrt{-1}\eta_\delta^2\tilde{f}^{q}\ol\pa\tilde{f}\wedge\pa\omega^{n-1}_{\ol{g}}
\\&\leq\int_{\ol{X}}(C_1\eta_\delta\tilde{f}^q|\pa\eta_\delta||\pa\tilde{f}|+C_2\eta_\delta^2\tilde{f}^{q}|\pa\tilde{f}|-qC_3\eta_\delta^2\tilde{f}^{q-1}|\pa\tilde{f}|^2)\omega_{\ol{g}}^{n}
\end{split}\end{equation}
where and henceforth, $C_i(i=1,2,3...)$ are suitable uniform constants. By the Cauchy inequality, it holds for any positive $\epsilon$,
\begin{equation}\begin{split}
C_1\eta_\delta\tilde{f}^q|\pa\eta_\delta||\pa\tilde{f}|
\leq\frac{\epsilon}{2}C_1\eta_\delta^2\tilde{f}^{q-1}|\pa\tilde{f}|^2+\frac{C_1}{2\epsilon}\tilde{f}^{q+1}|\pa\eta_\delta|^2,
\end{split}\end{equation}
\begin{equation}\begin{split}
C_2\eta_\delta^2\tilde{f}^{q}|\pa\tilde{f}|
\leq\frac{\epsilon}{2}C_2\eta_\delta^2\tilde{f}^{q-1}|\pa\tilde{f}|^2
+\frac{C_2}{2\epsilon}\eta_\delta^2\tilde{f}^{q+1}.
\end{split}\end{equation}
Taking $\epsilon=\frac{qC_3}{C_1+C_2}$, it follows
\begin{equation}\begin{split}
\frac{qC_3}{2}\int_{\ol{X}}\eta_\delta^2\tilde{f}^{q-1}|\pa\tilde{f}|^2\omega_{\ol{g}}^{n}
&\leq(A+\frac{C_2(C_1+C_2)}{2qC_3})\int_{\ol{X}}\eta_\delta^2\tilde{f}^{q+1}\omega_{\ol{g}}^{n}
\\&+\frac{C_1(C_1+C_2)}{2qC_3}\int_{\ol{X}}\tilde{f}^{q+1}|\pa\eta_\delta|^2\omega_{\ol{g}}^{n},
\end{split}\end{equation}
and therefore
\begin{equation}\begin{split}
\int_{\ol{X}}|\pa(\eta_\delta\tilde{f}^{\frac{q+1}{2}})|^2\dvol_{\ol{g}}
&=\int_{\ol{X}}\tilde{f}^{q+1}|\pa\eta_\delta|^2+\eta_\delta^2|\pa\tilde{f}^{\frac{q+1}{2}}|^2\dvol_{\ol{g}}
\\&=\int_{\ol{X}}\tilde{f}^{q+1}|\pa\eta_\delta|^2\dvol_{\ol{g}}+\frac{(q+1)^2}{4}\int_{\ol{X}}\eta_\delta^2\tilde{f}^{q-1}|\pa\tilde{f}|^2\dvol_{\ol{g}}
\\&\leq(\frac{q+1}{q})^2(C_4+qC_5)\int_{\ol{X}}\eta_\delta^2\tilde{f}^{q+1}\dvol_{\ol{g}}
\\&+C_6(\frac{q+1}{q})^2\int_{\ol{X}}\tilde{f}^{q+1}|\pa\eta_\delta|^2\dvol_{\ol{g}}.
\end{split}\end{equation}
\par Using the Sobolev inequality, we have
\begin{equation}\begin{split}
(\int_{\ol{X}}|\eta_\delta\tilde{f}^{\frac{q+1}{2}}|^{\frac{2n}{n-1}}\dvol_{\ol{g}})^{\frac{n-1}{n}}
&\leq C_7\int_{\ol{X}}\eta_\delta^2\tilde{f}^{q+1}+2|\pa(\eta_\delta\tilde{f}^{\frac{q+1}{2}})|^2\dvol_{\ol{g}}
\\&\leq C_8(q+1)^2\int_{\ol{X}}(\eta_\delta^2+|\pa\eta_\delta|^2)\tilde{f}^{q+1}\dvol_{\ol{g}}.
\end{split}\end{equation}
Hence as $\delta\rightarrow0$, it holds
\begin{equation}\begin{split}
(\int_{X}\tilde{f}^{\frac{n}{n-1}(q+1)}\dvol_{g})^{\frac{n-1}{n}}
&\leq C_8(q+1)^2\int_{X}\tilde{f}^{q+1}\dvol_{g},
\end{split}\end{equation}
and then we finish the proof via an iteration procedure.
\end{proof}
\section{Multiplicities of the prescribed curvature equation on $\mathbb{R}^2$}
In the following, we write $g_0=dx\otimes dx+dy\otimes dy$ for the standard metric.
\begin{lemma}\label{Alemma1}
If $\alpha, \phi\in C^\infty(\mathbb{R}^2)$ satisfy
\begin{enumerate}
\item $\alpha>0$ and $\alpha(z^{-1})|z|^{-4}$ extends to a positive smooth function on $\mathbb{R}^2$,
\item $\phi\geq0$, $\alpha\phi^{-1}\in L^\infty(\mathbb{R}^2)$ and $\alpha^{1-q}\phi^q\in L^1(\mathbb{R}^2)$ for some $q>1$,
\end{enumerate}
then $(\mathbb{R}^2,g_0)$ satisfies the Main Assumption with $\phi_{\mathbb{R}^2}=\phi$.
\end{lemma}
\begin{proof}
Let us consider $\mathbb{CP}^1$ as a compactification  of $\mathbb{C}\cong\mathbb{R}^2$, denote
\begin{equation}\begin{split}
U_1=\{[(z_0,z_1)]\in\mathbb{CP}^1,z_0\neq0\},\
U_2=\{[(z_0,z_1)]\in\mathbb{CP}^1,z_1\neq0\},
\end{split}\end{equation}
and coordinate functions $z=z_1^{-1}z_2,\ w=z_2^{-1}z_1$ on $U_1$, $U_2$ respectively. We set
\begin{equation}\begin{split}
\omega_g=\sqrt{-1}\alpha dz\wedge d\ol{z}
=\sqrt{-1}\alpha(w^{-1})|w|^{-4}dw\wedge d\ol{w},
\end{split}\end{equation}
which it is a fundamental form of a Hermitian metric $g$ on $\mathbb{CP}^1$ and it holds
\begin{equation}\begin{split}
\omega_{g}=2\alpha\omega_{g_0},\
\sqrt{-1}\Lambda_{\omega_g}\pa\ol\pa=2^{-1}\alpha^{-1}\sqrt{-1}\Lambda_{\omega_{g_0}}\pa\ol\pa.
\end{split}\end{equation}
\par Assume $f\in C^\infty(\mathbb{\mathbb{R}}^2)$ is nonnegative and bounded with
\begin{equation}\begin{split}
\sqrt{-1}\Lambda_{\omega_{g_0}}\pa\ol\pa f\geq-A\phi,
\end{split}\end{equation}
for a positive constant $A$, then it is equivalent to
\begin{equation}\begin{split}
\sqrt{-1}\Lambda_{\omega_{g}}\pa\ol\pa f\geq-2^{-1}A\alpha^{-1}\phi.
\end{split}\end{equation}
We compute
\begin{equation}\begin{split}
||\phi||_{L^1(\mathbb{R}^2)}
&=||(2\alpha)^{-1}\phi||_{L^1(\mathbb{R}^2,g)}
\\&\leq||(2\alpha)^{-1}\phi||_{L^q(\mathbb{R}^2,g)}\vol(\mathbb{R}^2,g)^{\frac{q-1}{q}}
\\&=||(2\alpha)^{1-q}\phi^q||_{L^1(\mathbb{R}^2)}^{\frac{1}{q}}\vol(\mathbb{R}^2,g)^{\frac{q-1}{q}}
\\&<\infty,
\end{split}\end{equation}
by \cite[Proposition 2.2]{Si1988} we conclude the following inequality holds weakly on $\mathbb{CP}^1$,
\begin{equation}\begin{split}
\sqrt{-1}\Lambda_{\omega_{g}}\pa\ol\pa f\geq-2^{-1}A\alpha^{-1}\phi.
\end{split}\end{equation}
Therefore there is a constant $C$ independent of $f$ such that
\begin{equation}\begin{split}
f&\leq C(1+\int_{\mathbb{CP}^1}f\omega_g)
\\&=C(1+2\int_{\mathbb{R}^2}f\alpha\dvol_{g_0})
\\&\leq C(1+2\sup\limits_{\mathbb{R}^2}\alpha\phi^{-1}\int_{\mathbb{R}^2}f\phi\dvol_{g_0})
\\&\leq(C+2\sup\limits_{\mathbb{R}^2}\alpha\phi^{-1})(1+\int_{\mathbb{R}^2}f\phi\dvol_{g_0}),
\end{split}\end{equation}
where we have used $\alpha^{-1}\phi\in L^q(\mathbb{R}^2,g)$ for some $q>1$.
\end{proof}
If we take $\alpha=(1+|z|^2)^{-2}$, $\phi=\phi_p=(1+|z|^2)^{-p}$, it follows
\begin{corollary}\label{Acor1}
$(\mathbb{R}^2,g_0)$ satisfies the Main Assumption with $\phi_{\mathbb{R}^2}=\phi_{p}$ and $p\in(1,2]$.
\end{corollary}
The following lemma demonstrates that $\mathbb{R}^2$ can also satisfy our Main Assumption when the underlying metric is not the standard one.
\begin{lemma}\label{Alemma2}
If $\alpha, \phi, \psi\in C^\infty(\mathbb{R}^2)$ satisfy
\begin{enumerate}
\item $\alpha>0$ and $\alpha(z^{-1})|z|^{-4}$ extends to a positive smooth function on $\mathbb{R}^2$,
\item $\phi>0$, $\psi\geq0$, $\alpha\phi^{k}\psi^{-1}\in L^\infty(\mathbb{R}^2)$ and $\alpha^{1-q}(\phi^{-k}\psi)^{q}\in L^1(\mathbb{R}^2)$ for some $q>1$,
\end{enumerate}
then $(\mathbb{R}^2,g_k)$ satisfies the Main Assumption with $\phi_{\mathbb{R}^2}=\psi$, where $g_k=\phi^{-k}g_0$.
\end{lemma}
\begin{proof}
We use the notations in Lemma \ref{Alemma1} and note
\begin{equation}\begin{split}
||\psi||_{L^1(\mathbb{R}^2,g_k)}
&=||\phi^{-k}\psi||_{L^1(\mathbb{R}^2)}
\\&=||(2\alpha)^{-1}\phi^{-k}\psi||_{L^1(\mathbb{R}^2,g)}
\\&\leq||(2\alpha)^{-1}\phi^{-k}\psi||_{L^q(\mathbb{R}^2,g)}\vol(\mathbb{R}^2,g)^{\frac{q-1}{q}}
\\&=||(2\alpha)^{1-q}(\phi^{-k}\psi)^q||_{L^1(\mathbb{R}^2)}^{\frac{1}{q}}\vol(\mathbb{R}^2,g)^{\frac{q-1}{q}}
\\&<\infty.
\end{split}\end{equation}
Assume $f\in C^\infty(\mathbb{\mathbb{R}}^2)$ is nonnegative and bounded with
\begin{equation}\begin{split}
\sqrt{-1}\Lambda_{\omega_{g_k}}\pa\ol\pa f\geq-A\psi,
\end{split}\end{equation}
for a positive constant $A$, then
\begin{equation}\begin{split}
\sqrt{-1}\Lambda_{\omega_{g_0}}\pa\ol\pa f\geq-A\phi^{-k}\psi.
\end{split}\end{equation}
Therefore there is a uniform constant $C$ such that
\begin{equation}\begin{split}
f
&\leq(C+2\sup_{\mathbb{R}^2}\alpha\phi^k\psi^{-1})(1+\int_{\mathbb{R}^2}f\phi^{-k}\psi\dvol_{g_0})
\\&=(C+2\sup_{\mathbb{R}^2}\alpha\phi^k\psi^{-1})(1+\int_{\mathbb{R}^2}f\psi\dvol_{g_k}).
\end{split}\end{equation}
\end{proof}
Let $\alpha=(1+|z|^2)^{-2}$, $\phi=(1+|z|^2)^{-1}$ and $\psi=\psi_l=(1+|z|^2)^{-l}$, it is easy to see
\begin{corollary}\label{Acor2}
$(\mathbb{R}^2,g_k)$ satisfies the Main Assumption with $\phi_{\mathbb{R}^2}=\psi_l$, where $g_k=\phi^{-k}g_0$ and $l>0$, $k\in[l-2,l-1)$ are two constants.
\end{corollary}
Let us consider $(\mathbb{R}^2,g_k)$ and $\phi_{\mathbb{R}^2}$ given in Lemma \ref{Alemma2}. Using $(\ref{conformal})$, we have
\begin{equation}\begin{split}\label{Sch1}
S_{g_k}^{Ch}
&=k\phi^{k}\sqrt{-1}\Lambda_{\omega_{g_0}}\pa\ol\pa\log\phi
=\frac{k}{2}\phi^{k}\Delta\log\phi.
\end{split}\end{equation}
In view of Theorem \ref{KWthm}, we have
\begin{proposition}\label{Aprop1}
Let $\alpha, \phi, \psi\in C^\infty(\mathbb{R}^2)$ be three functions satisfying
\begin{enumerate}
\item $\phi>0$ and $\psi\geq0$,
\item $\alpha>0$ and $\alpha(z^{-1})|z|^{-4}$ extends to a positive smooth function on $\mathbb{R}^2$.
\end{enumerate}
If $K\in C^\infty(\mathbb{R}^2)$ is nonpositive and nonzero with decay $|K|\leq\Lambda\psi$ for a constant $\Lambda$, then for any constant $k$ such that
\begin{enumerate}
\item $\alpha\phi^{k}\psi^{-1}\in L^\infty(\mathbb{R}^2)$ and $\alpha^{1-q}(\phi^{-k}\psi)^{q}\in L^1(\mathbb{R}^2)$ for some $q>1$,
\item $\int_{\mathbb{R}^2}k\Delta\log\phi<0$ and $|k\phi^{k}\Delta\log\phi|\leq2\Lambda\psi$,
\end{enumerate}
the prescribed curvature equation
\begin{equation}\begin{split}\label{multipleeq1}
\sqrt{-1}\Lambda_{\omega_{g_k}}\pa\ol\pa v+Ke^{v}=S^{Ch}_{g_k}
\end{split}\end{equation}
possesses a solution $v_k\in C^\infty(\mathbb{R}^2)\cap L^\infty(\mathbb{R}^2)$ with $|dv_k|\in L^2(\mathbb{R}^2)$.
\end{proposition}
Since $v$ satisfies $(\ref{multipleeq1})$ if and only if the metric
\begin{equation}\begin{split}
e^{v}g_k=e^{u}g_0,\ u=v-k\log\phi,
\end{split}\end{equation}
has Chern scalar curvature $K$, which is also equivalent to
\begin{equation}\begin{split}\label{multipleeq2}
\sqrt{-1}\Lambda_{\omega_{g_0}}\pa\ol\pa u+Ke^{u}=0.
\end{split}\end{equation}
\begin{theorem}\label{Athm1}
Let $\alpha,\phi,\psi\in C^\infty(\mathbb{R}^2)$ be three functions satisfying
\begin{enumerate}
\item $\phi>0$ and $\psi\geq0$,
\item $\alpha>0$ and $\alpha(z^{-1})|z|^{-4}$ extends to a positive smooth function on $\mathbb{R}^2$.
\end{enumerate}
If $K\in C^\infty(\mathbb{R}^2)$ is nonpositive and nonzero with decay $|K|\leq\Lambda\psi$ for a constant $\Lambda$, then for any constant $k$ such that
\begin{enumerate}
\item $\alpha\phi^{k}\psi^{-1}\in L^\infty(\mathbb{R}^2)$ and $\alpha^{1-q}(\phi^{-k}\psi)^{q}\in L^1(\mathbb{R}^2)$ for some $q>1$,
\item $\int_{\mathbb{R}^2}k\Delta\log\phi<0$ and $|k\phi^{k}\Delta\log\phi|\leq2\Lambda\psi$,
\end{enumerate}
$(\ref{eqR2})$ possesses a solution on $\mathbb{R}^2$ and it has the following form
\begin{equation}\begin{split}
2u_k=v_k-k\log\phi,
\end{split}\end{equation}
where $v_k\in C^\infty(\mathbb{R}^2)\cap L^\infty(\mathbb{R}^2)$ with $|dv_k|\in L^2(\mathbb{R}^2)$.
\end{theorem}
\begin{rem}\label{00000}
Let $k_1$, $k_2$ be two constants satisfying the conditions of $k$, then 
\begin{equation}\begin{split}
C_1-(k_1-k_2)\log\phi\leq
2(u_{k_1}-u_{k_2})
\leq C_2-(k_1-k_2)\log\phi,
\end{split}\end{equation}
where $C_1=\inf\limits_{\mathbb{R}^2}(v_{k_1}-v_{k_2})$ and $C_2=\sup\limits_{\mathbb{R}^2}(v_{k_1}-v_{k_2})$.
\end{rem}
We also have the uniqueness.
\begin{theorem}\label{Athm2}
Any solution $u$ to $(\ref{eqR2})$ which takes the form
\begin{equation}\begin{split}
2u=v-k\log\phi,
\end{split}\end{equation}
with $v\in C^\infty(\mathbb{R}^2)\cap L^\infty(\mathbb{R}^2)$, $k$ and $\phi$ given in Theorem \ref{Athm1}, is exactly $u_k$.
\end{theorem}
\begin{proof}
One sees that $v$ and $v_k$ satisfy
\begin{equation}\begin{split}
\sqrt{-1}\Lambda_{\omega_{g_k}}\pa\ol\pa v+Ke^{v}=S^{Ch}_{g_k},
\end{split}\end{equation}
\begin{equation}\begin{split}
\sqrt{-1}\Lambda_{\omega_{g_k}}\pa\ol\pa v_k+Ke^{v_k}=S^{Ch}_{g_k},
\end{split}\end{equation}
where $g_k=\phi^{-k}g_0$. We compute
\begin{equation}\begin{split}
\sqrt{-1}\Lambda_{g_k}\pa\ol\pa(e^{v-v_{k}}+e^{v-v_{k}})
&=(e^{v-v_{k}}-e^{v_{k}-v})
\sqrt{-1}\Lambda_{g_k}\pa\ol\pa(v-v_{k})
\\&+(e^{v-v_{k}}+e^{v_{k}-v})|\pa(v-v_k)|^2_{g_k}
\\&\geq-K(e^{v-v_k}-e^{v_k-v})
(e^v-e^{v_k})
\\&\geq0.
\end{split}\end{equation}
Namely, $e^{v-v_{k}}+e^{v-v_{k}}$ is a bounded sub-harmonic function on $\mathbb{R}^2$ and hence constant by standard Liouville theorem, then it follows $u=u_k$.
\end{proof}
If we consider $(\mathbb{R}^2,g_k)$ and $\phi_{\mathbb{R}^2}$ given in Corollary \ref{Acor2}, we have
\begin{equation}\begin{split}\label{Sch1}
k\phi^{k}\Delta\log\phi
&=4k\phi^{k}\frac{\pa^2\log\phi}{\pa z\pa\ol{z}}
\\&=\frac{-4k}{(1+|z|^2)^{k}}(\frac{\frac{\pa^2|z|^2}{\pa z\pa\ol{z}}}{1+|z|^2}-\frac{|\pa |z|^2|^2}{(1+|z|^2)^2})
\\&=\frac{4k}{(1+|z|^2)^{k}}(\frac{|z|^2}{(1+|z|^2)^2}-\frac{1}{1+|z|^2})
\\&=\frac{-4k}{(1+|z|^2)^{k+2}}.
\end{split}\end{equation}
If $k>0$ and $\Lambda\geq2k$, we have
\begin{equation}\begin{split}\label{Sch2}
\int_{\mathbb{R}^2}k\Delta\log\phi<0,\
|k\phi^{k}\Delta\log\phi|
\leq\frac{4k}{(1+|z|^2)^{l}}
=2\Lambda\phi_{\mathbb{R}^2}.
\end{split}\end{equation}
\par It follows from Theorem \ref{Athm1}, Theorem \ref{Athm2} and above that
\begin{proposition}\label{Aprop2}
If $K\in C^\infty(\mathbb{R}^2)$ is nonpositive and nonzero with decay $|K|\leq\Lambda(1+|z|^2)^{-l}$ for two constants $l>1$ and $\Lambda$. Then for any constant
\begin{equation}\begin{split}
k\in[l-2,l-1)\cap(0,\frac{\Lambda}{2}],
\end{split}\end{equation}
$(\ref{eqR2})$ possesses a solution on $\mathbb{R}^2$ and it has the following form
\begin{equation}\begin{split}
2u_k=v_k+k\log(1+|z|^2),
\end{split}\end{equation}
with $v_k\in C^\infty(\mathbb{R}^2)\cap L^\infty(\mathbb{R}^2)$ and $|dv_k|\in L^2(\mathbb{R}^2)$. On the other hand, any solution $u$ of the form
\begin{equation}\begin{split}
2u=v+k\log(1+|z|^2),
\end{split}\end{equation}
with $v\in C^\infty(\mathbb{R}^2)\cap L^\infty(\mathbb{R}^2)$ and constant $k\in[l-2,l-1)\cap(0,\frac{\Lambda}{2}]$, is exactly $u_k$.
\end{proposition}
\begin{rem}
By Remark \ref{00000}, $u_{k_1}\neq u_{k_2}$ for any different $k_1,k_2\in[l-2,l-1)\cap(0,\frac{\Lambda}{2}]$.
\end{rem}
The geometric interpretation is as follows.
\begin{corollary}
If $K\in C^\infty(\mathbb{R}^2)$ is nonpositive and nonzero with decay $|K|\leq\Lambda(1+|z|^2)^{-l}$ for two constants $l>1$ and $\Lambda$. Then for any constant
\begin{equation}\begin{split}
k\in[l-2,l-1)\cap(0,\frac{\Lambda}{2}],
\end{split}\end{equation}
there exists a conformal metric $\tilde{g}_k=e^{v_k}(1+|z|^2)^{k}g_0$ on $\mathbb{R}^2$ such that $S^{Ch}_{\tilde{g}_k}=K$, where $v_k\in C^\infty(\mathbb{R}^2)\cap L^\infty(\mathbb{R}^2)$ and $|dv_k|\in L^2(\mathbb{R}^2)$.
\end{corollary}
\section{Vortex equation on holomorphic line bundles}
Assume $L$ is a holomorphic line bundle over $X$ and $\vp\in A^0(L)$ is a nonzero section. Consider two Hermitian structures $H$ and $K$, related by  $H=e^fK$, then
\begin{equation}\begin{split}
\sqrt{-1}\Lambda_{\omega_g}F_H+\frac{1}{2}\vp\otimes\vp^{\ast H}
&=\sqrt{-1}\Lambda_{\omega_g} F_K-\sqrt{-1}\Lambda_{\omega_g}\pa\ol\pa f+\frac{1}{2}|\vp|^2e^f,
\end{split}\end{equation}
where $F_\bullet$ denotes the Chren curvature of $\bullet$. So the solvability of the vortex equation
\begin{equation}\begin{split}\label{Vortex1}
\sqrt{-1}\Lambda_{\omega_g}F_H+\frac{1}{2}\vp\otimes\vp^{\ast H}=\frac{\lambda}{2},\ \lambda\in C^\infty(X),
\end{split}\end{equation}
is equivalent to the solvability of
\begin{equation}\begin{split}\label{Vortex2}
\sqrt{-1}\Lambda_{\omega_g}\pa\ol\pa f-\frac{1}{2}|\vp|^2e^f
=\sqrt{-1}\Lambda_{\omega_g}F_K-\frac{\lambda}{2}.
\end{split}\end{equation}
\par Hence Theorem \ref{vortexthm} below is a consequence of the proof of Theorem \ref{KWthm}.
\begin{theorem}\label{vortexthm}
Given a Hermitian structure $K$ on $L$ such that
\begin{equation}\begin{split}
\int_X(2\sqrt{-1}\Lambda_{\omega_g} F_K-\lambda)\dvol_g<0,\ 
|\sqrt{-1}\Lambda_{\omega_g}F_K-\frac{\lambda}{2}|+|\vp|^2\leq\Lambda\phi_X, \end{split}\end{equation}
for a constant $\Lambda$, then $(\ref{Vortex1})$ admits a solution.
\end{theorem}
\begin{corollary}\label{vortex00}
Let $X=\ol{X}\setminus\Sigma$ be the complement of a complex analytic subset $\Sigma$ in a compact complex manifold $\ol{X}$ with complex codimension at least two and $g=\ol{g}|_{X}$ for a Gauduchon metric $\ol{g}$ on $\ol{X}$. Given a Hermitian structure $K$ on $L$ such that
\begin{equation}\begin{split}
\int_{X}(2\sqrt{-1}\Lambda_{\omega_g} F_K-\lambda)\dvol_g<0,\
|\sqrt{-1}\Lambda_{\omega_g}F_K-\frac{\lambda}{2}|+|\vp|^2\in L^\infty(X),
\end{split}\end{equation}
then $(\ref{Vortex1})$ admits a solution.
\end{corollary}
\begin{proposition}\label{vortex1}
For two constants $l>0$ and $k\in[l-2,l-1)$, we set
\begin{equation}\begin{split}
\omega_k=\frac{\sqrt{-1}}{2}(1+|z|^2)^{k}dz\wedge d\ol{z}
\end{split}\end{equation}
and assume $f,h$ are smooth nonzero functions on $\mathbb{C}$ satisfying
\begin{equation}\begin{split}
\int_{\mathbb{C}}f\omega_k<0,\ |f|\leq\Lambda(1+|z|^2)^{-l},\ -\Lambda(1+|z|^2)^{-l}\leq h\leq0,
\end{split}\end{equation}
for a constant $\Lambda$, then the equation
\begin{equation}\begin{split}
\sqrt{-1}\Lambda_{\omega_k}\pa\ol\pa u+he^u=f
\end{split}\end{equation}
admits a bounded smooth function on $\mathbb{C}$.
\end{proposition}
\begin{proof}
It follows from Theorem \ref{KWthm} and Corollary \ref{Acor2}.
\end{proof}
\begin{proposition}\label{vortex2}
Let $M$ be a compact complex manifold with Gauduchon metric $g_M$. Assume $f,h$ are smooth nonzero functions on $M\times\mathbb{C}$ satisfying
\begin{equation}\begin{split}
\int_{M\times\mathbb{C}}f\dvol_{g_M\times g_0}<0,\
|f|\leq\Lambda(1+|z|^2)^{-l},\
-\Lambda(1+|z|^2)^{-l}\leq h\leq0,
\end{split}\end{equation}
at $(m,z)\in M\times\mathbb{C}$ for two constants $l\in(1,2]$ and $\Lambda$, then
\begin{equation}\begin{split}
\sqrt{-1}\Lambda_{\omega_{g_M\times g_0}}\pa\ol\pa u+he^u=f
\end{split}\end{equation}
admits a bounded smooth function on $M\times\mathbb{C}$.
\end{proposition}
\begin{proof}
By Theorem \ref{KWthm}, it suffices to prove $(M\times\mathbb{C},g_M\times  g_0)$ satisfis the Main Assumption which can be directly checked by Moser iteration, see \cite{ZZ2021}. Below we sketch the proof just for readers' conveniences. If $f\in C^\infty(M\times\mathbb{C})$ is nonnegative, bounded and 
\begin{equation}\begin{split}\label{prop421}
\sqrt{-1}\Lambda_{\omega_{g_M\times g_0}}\pa\ol\pa f\geq-A(1+|z|^2)^{-l},
\end{split}\end{equation}
for a positive constant $A$, we set $\tilde{f}(z)=\int_Mf(m,z)\dvol_{g_M}$ and it follows
\begin{equation}\begin{split}\label{prop422}
\sqrt{-1}\Lambda_{\omega_{g_0}}\pa\ol\pa\tilde{f}
&\geq-A\vol(M,g_M)(1+|z|^2)^{-l},
\end{split}\end{equation}
from which and Corollary \ref{Acor1} it holds
\begin{equation}\begin{split}\label{prop423}
\sup\limits_{\mathbb{C}}\int_Mf(m,z)\dvol_{g_M}\leq C_1\left(1+\int_{M\times\mathbb{C}}f(1+|z|^2)^{-l}\dvol_{g_M\times g_0}\right),
\end{split}\end{equation}
for a positive constant $C_1$. In addition, a Moser iteration procedure on $(\ref{prop421})$ yields
\begin{equation}\begin{split}\label{prop424}
\sup\limits_{M\times B_1(z_0)}f\leq C_2(1+\int_{M\times B_2(z_0)}f\dvol_{g_M}),
\end{split}\end{equation}
for any $z_0\in\mathbb{C}$ and a positive constant $C_2$, where $B_1(z_0)=\{z\in\mathbb{C}, |z-z_0|<1\}$, $B_2(z_0)=\{z\in\mathbb{C}, |z-z_0|<2\}$. By $(\ref{prop423})$ and $(\ref{prop424})$, we have
\begin{equation}\begin{split}\label{prop425}
\sup\limits_{M\times B_1(z_0)}f\leq C_3(1+\int_{M\times\mathbb{C}}f(1+|z|^2)^{-l}\dvol_{g_M\times g_0}),
\end{split}\end{equation}
for a positive constant $C_3$ and the proof is complete since $z_0$ is arbitrary.
\end{proof}
Applying Proposition \ref{vortex1} and Proposition \ref{vortex2} on $(\ref{Vortex2})$, we have
\begin{corollary}\label{vortex11}
Let $	X=\mathbb{C}$ and given a Hermitian structure $K$ on $L$, two constants $k$, $l$ such that $l>0$, $k\in[l-2,l-1)$ and
\begin{equation}\begin{split}
\int_{\mathbb{C}}(2\sqrt{-1}F_K-\lambda\omega_{k})<0,\
|\sqrt{-1}\Lambda_{\omega_{k}}F_K-\frac{\lambda}{2}|+|\vp|^2\leq\Lambda(1+|z|^2)^{-l},
\end{split}\end{equation}
where $\omega_k=\frac{\sqrt{-1}}{2}(1+|z|^2)^{k}dz\wedge d\ol{z}$, then the curvature equation
\begin{equation}\begin{split}
\sqrt{-1}F_H+\frac{1}{2}\vp\otimes\vp^{\ast H}\omega_{k}=\frac{\lambda}{2}\omega_{k}
\end{split}\end{equation}
admits a solution.
\end{corollary}
\begin{corollary}\label{vortex22}
Let $X=M\times\mathbb{C}$, where $M$ is a compact complex manifold with Gauduchon metric $g_M$. Given a Hermitian structure $K$ on $L$ such that
\begin{equation}\begin{split}
\int_{M\times\mathbb{C}}(2\sqrt{-1}\Lambda_{\omega_{g_M\times g_0}}F_K-\lambda)\dvol_{g_M\times g_0}<0,
\end{split}\end{equation}
\begin{equation}\begin{split}
|\sqrt{-1}\Lambda_{\omega_{g_M\times g_0}}F_K-\frac{\lambda}{2}|+|\vp|^2\leq\Lambda(1+|z|^2)^{-l},
\end{split}\end{equation}
at $(m,z)\in M\times\mathbb{C}$ for two constants $l\in(1,2]$ and $\Lambda$, then 
\begin{equation}\begin{split}
\sqrt{-1}\Lambda_{\omega_{g_M\times g_0}}F_H+\frac{1}{2}\vp\otimes\vp^{\ast H}=\frac{\lambda}{2}
\end{split}\end{equation}
admits a solution.
\end{corollary}


\begin{thebibliography}{1}
\small
\bibitem{Ah1938}L.V. Ahlfors, An extension of Schwarz's lemma. Trans. Amer. Math. Soc. \textbf{43}(1938), no. 3, 359-364.
\bibitem{ACS2017}D. Angella, S. Calamai and C. Spotti, On the Chern-Yamabe problem. Math. Res. Lett. \textbf{24}(2017), no. 3, 645-677.
\bibitem{AF2022}D. Angella and P. Francesco, On the linearization stability of the Chern-scalar curvature. Math. Z. \textbf{301}(2022), no. 2, 1675-1693.
\bibitem{Av1986}P. Aviles, Conformal complete metrics with prescribed nonnegative Gaussian curvature in $\mathbb{R}^2$. Invent. Math. \textbf{83}(1986), no. 3, 519-544.
\bibitem{Br1990}S.B. Bradlow, Vortices in holomorphic line bundles over closed K\"{a}hler manifolds. Comm. Math. Phys. \textbf{135}(1990), no. 1, 1-17.
\bibitem{Br1991}S.B. Bradlow, Special metrics and stability for holomorphic bundles with global sections. J. Differential Geom. \textbf{33}(1991), no. 1, 169-213.
\bibitem{CZ2020}S. Calamai and F. Zou, A note on Chern-Yamabe problem. Differential Geom. Appl. \textbf{69}(2020), 101612, 14 pp.
\bibitem{CK2000}S. Chanillo and M. Kiessling, Surfaces with prescribed Gauss curvature. Duke Math. J. \textbf{105}(2000), no. 2, 309-353.
\bibitem{CL1991}W.X. Chen and C. Li, Classification of solutions of some nonlinear elliptic equations. Duke Math. J. \textbf{63}(1991), no. 3, 615-622.
\bibitem{CL1993}W.X. Chen and C. Li, Qualitative properties of solutions to some nonlinear elliptic equations in $\mathbb{R}^2$. Duke Math. J. \textbf{71}(1993), no. 2, 427-439.
\bibitem{CL1997}K.S. Cheng and C.S. Lin, On the asymptotic behavior of solutions of the conformal Gaussian curvature equations in $\mathbb{R}^2$. Math. Ann. \textbf{308}(1997), no. 1, 119-139.
\bibitem{CL2000}K.S. Cheng and C.S. Lin, Conformal metrics with prescribed nonpositive Gaussian curvature on $\mathbb{R}^2$. Calc. Var. Partial Differential Equations. \textbf{11}(2000), no. 2, 203-231.
\bibitem{CN1991a}K.S. Cheng and W.M. Ni, On the structure of the conformal Gaussian curvature equation on $\mathbb{R}^2$. Duke Math. J. \textbf{62}(1991), no. 3, 721-737.
\bibitem{CN1991b}K.S. Cheng and W.M. Ni, On the structure of the conformal Gaussian curvature equation on $\mathbb{R}^2$. \uppercase\expandafter{\romannumeral2}. Math. Ann. \textbf{290}(1991), no. 4, 671-680.
\bibitem{DL2017}M.G. Dabkowski and M.T. Lock, An equivalence of scalar curvatures on Hermitian manifolds. J. Geom. Anal. \textbf{27}(2017), no. 1, 239-270.
\bibitem{Fu2022}E. Fusi, The prescribed Chern scalar curvature problem. J. Geom. Anal. \textbf{32}(2022), no. 6, Paper No. 187, 21 pp.
\bibitem{Ga1977}P. Gauduchon, Le th\'{e}or\`{e}me de l'excentricit\'{e} nulle. (French) C. R. Acad. Sci. Paris S\'{e}r. A-B. \textbf{285}(1977), no. 5, A387-A390.
\bibitem{Ho2021}P.T. Ho, Results related to the Chern-Yamabe flow. J. Geom. Anal. \textbf{31}(2021), no. 1, 187-220.
\bibitem{HS2021}P.T. Ho and J. Shin, Chern-Yamabe problem and Chern-Yamabe solition. Internat. J. Math. \textbf{32}(2021), no. 3, Paper No. 2150016, 22 pp.
\bibitem{KW1974a}J.L. Kazdan and F.W. Warner, Curvature functions for compact $2$-manifolds. Ann. of Math. (2). \textbf{99}(1974), 14-47.
\bibitem{KW1974b}J.L. Kazdan and F.W. Warner, Curvature functions for open 2-manifolds. Ann. of Math. (2). \textbf{99}(1974), 203-219.
\bibitem{LM2018}M. Lejmi and A. Mehdi, On the Chern-Yamabe flow. J. Geom. Anal. \textbf{28}(2018), no. 3, 2692-2706.
\bibitem{Li2000}C.S. Lin, Uniqueness of conformal metrics with prescribed total curvature in $\mathbb{R}^2$. Calc. Var. Partial Differential Equations. \textbf{10}(2000), no. 4, 291-319.
\bibitem{LY2005}P. Liu and C. Yao, Vortex equation in holomorphic line bundle over complex manifold. J. Math. Anal. Appl. \textbf{311}(2005), no. 1, 352-356.
\bibitem{Mc1984}R.C. McOwen, On the equation $\Delta u+Ke^{2u}=f$ and prescribed nagative curvature in $\mathbb{R}^2$. J. Math. Anal. Appl. \textbf{103}(1984), no. 2, 365-370.
\bibitem{Mc1985}R.C. McOwen, Conformal metrics in $\mathbb{R}^2$ with prescribed Gaussian curvature and positive total curvature. Indiana Univ. Math. J. \textbf{34}(1985), no. 1, 97-104.
\bibitem{Mo2020}T. Mochizuki, Kobayashi-Hitchin correspondence for analytically stable bundles. Trans. Amer. Math. Soc. \textbf{373}(2020), no. 1, 551-596.
\bibitem{Ni1982}W.M. Ni, On the elliptic equation $\Delta u+K(x)e^{2u}=0$ and conformal metrics with prescribed Gaussian curvatures. Invent. Math. \textbf{66}(1982), no. 2, 343-352.
\bibitem{Sa1972}D.H. Sattinger, Conformal metrics in $\mathbb{R}^2$ with prescirbed curvature. Indiana Univ. Math. J. \textbf{22}(1972/1973), 1-4.
\bibitem{Si1988}C.T. Simpson, Constructing variations of Hodge structure using Yang-Mills theory and applications to uniformization. J. Amer. Math. Soc. \textbf{1}(1988), no. 4, 867-918.
\bibitem{WZ2021}R. Wang and P. Zhang, The Hitchin-Kobayashi correspondence for holomorphic pairs over non-K\"{a}hler manifolds. Bull. Sci. Math. \textbf{172}(2021), Paper No. 103050, 32 pp.
\bibitem{WZ2008}Y. Wang and X. Zhang, A class of Kazdan-Warner typed equations on non-compact Riemannian manifolds. Sci. China Ser. A \textbf{51}(2008), no. 6, 1111-1118.
\bibitem{Ya1960}H. Yamabe, On a deformation of Riemannian structures on compact manifolds. Osaka Math. J. \textbf{12}(1960), 21-37.
\bibitem{Yu2023}W. Yu, Prescribed Chern scalar curvatures on compact Hermitian manifolds with negative Gauduchon degree. J. Funct. Anal. \textbf{285}(2023), no. 2, Paper No. 109948.
\bibitem{ZZ2021}C. Zhang and X. Zhang, Analytically stable Higgs bundles on some non-K\"{a}hler manifolds. Ann. Mat. Pura. Appl. (4). \textbf{200}(2021), no. 4, 1683-1707.
\end{thebibliography}
\end{document}